\documentclass{article}
\usepackage{graphicx} 
\usepackage{amsmath,amssymb,amsthm}
\usepackage{color}

\newcommand{\R}{\mathbb R}
\newcommand{\eps}{\varepsilon}
\newcommand{\dy}{\, \mathrm d y}
\newcommand{\pr}{\mathrm{Pr}}
\newcommand{\E}{\mathrm{E}}

\newtheorem{lemma}{Lemma}
\newtheorem{proposition}{Proposition}
\newtheorem{remark}{Remark}

\newcommand\blfootnote[1]{%
  \begingroup
  \renewcommand\thefootnote{}\footnote{#1}%
  \addtocounter{footnote}{-1}%
  \endgroup
}

\title{Corrections to and Improvements on Results from ``The Laplacian Spectrum of Large Graphs
Sampled from Graphons''}
\author{Federica~Garin \and Paolo~Frasca \and Renato~Vizuete}

\begin{document}

\maketitle
\blfootnote{F.~Garin and P.~Frasca are with Univ.\ Grenoble Alpes, Inria, CNRS, Grenoble INP, GIPSA-Lab, F-38000 Grenoble, France. R.~Vizuete is with ICTEAM institute, UCLouvain, Louvain-la-Neuve, Belgium. Emails: federica.garin@inria.fr, palo.frasca@gipsa-lab.fr, renato.vizueteharo@uclouvain.be}

\begin{abstract}In this note we correct the proof of Proposition~4 in our paper ``The Laplacian Spectrum of Large Graphs Sampled from Graphons'' and we improve several results therein. To this end, we prove a new concentration lemma about 
degrees and Laplacian eigenvalues. This lemma allows us to improve several bounds and to dispense from assuming that the graphon is bounded away from zero in several results. This extension leads, in particular, to correct the proof of Proposition~4. Additionally, we extend Proposition~4 to graphs that are sampled from graphons by using deterministic latent variables.
\end{abstract}

\section{Introduction}  \label{sect:typos}
In this note, we assume all 
notation and definitions 
from \cite{vizuete2021laplacian}.
Arabic numbering for lemmas, propositions, theorems, remarks, definitions and equations refers to results in \cite{vizuete2021laplacian}, while letters for equations and roman numbering for lemmas, propositions etc.\ refer to this note.

\section{Typos}
We begin by noticing that the following typos that should be corrected in \cite{vizuete2021laplacian}:
\begin{itemize}
    \item In Definition 2.2 (Large enough $N$), Eq.s~(4a)-(4b)  
    should be replaced by the following:
    \[2 \, b_N < \min_{k \in \{1, \dots, K+1\}} (\alpha_k - \alpha_{k-1}),  \]
    \[ \dfrac{1}{N}\log\left(\dfrac{2N}{\nu}\right)+b_N \, (2K+3L)<\max_{x}d_W(x) , \label{eq_condi2} \]
      where
    \[    b_N = \frac{1}{N} + \sqrt{\frac{8 \log (N/\nu)}{N+1}}  \,. \]
    Indeed, we had mistakenly used $b_N = \frac{1}{N}$, which is correct in the case of deterministic sampling, i.e., $X_i = i/N$, while we are considering stochastic sampling.
    \item  Eq.~(22) should be replaced by the following
    \[   \sqrt{\frac{\log(2N/\nu)}{N}} < 
        \frac{\eta_W^2}{1+2 \eta_W}  \,. \]
        Indeed, this is the right inequality that ensures that $\eta_W - \gamma_N>0$,
        so that the denominators in the upper bound of $| R^{\mathrm{ave}}_N-R_{W,N}^{\mathrm{ave}}|$ in Theorem~2 are all positive.
        
    We will see in Sect.~\ref{sect:improvements} that, when using tighter upper bounds proposed in this note, Eq.~(22) can be further replaced by a simpler expression.
\end{itemize}



\section{Dispensing from $W$ being bounded away from zero}
\subsection{Improved Lemma~3}  \label{subsect:Lemma3}
The key improvement is replacing Lemma~3 in \cite{vizuete2021laplacian} with the following result.
\begin{lemma}[Lemma~3, improved]  \label{lemma3-new}
Given a graphon $W$, 
with probability at least $1-\nu$ the normalized degrees  of the  graphs $G_N$ and $\bar G_N$ sampled from $W$ satisfy:
\begin{equation}\label{gamma_new}
\max_{i=1,\dots, N} \vert \delta_{(i)}-\bar \delta_{(i)}\vert \leq 
\sqrt{\dfrac{ \log(2N/\nu)}{N}}=:\gamma(N),
\end{equation}
and moreover,
 if $\bar d_{(N)}> \frac{4}{9} \log (2N/\nu)$, 
then with probability at least $1-2\nu$
the normalized eigenvalues of their Laplacian matrices
$L_N$ and $\bar L_N$ satisfy:
\begin{equation}\label{phi_new}
\max_{i=1,\dots, N}  \vert \mu_i-\bar\mu_i \vert\leq 
3 \, \sqrt{\frac{\log(2N/\nu)}{N}}
=:\varphi(N).
\end{equation}    
\end{lemma}

\begin{proof}
The proof follows the same main lines of the proof of Lemma 3, with some modifications. The whole proof is given below, for better readability.\\
For any $i$ such that $\bar d_i > 0$, we can 
use Chernoff bound, as in \cite[Proof of Theorem 2]{chung2011spectra},
thanks to the remark that $\bar d_i$ is the expectation of $d_i$, conditioned on $X_1, \dots, X_N$, and that $d_i$ is the $i$th row-sum of $A_N$.
The Chernoff bound gives us
$$
\pr\Big[\vert d_i-\bar d_i\vert>b\,\bar d_i\Big]\leq \dfrac{\nu}{N} \quad \mathrm{if} \quad b\geq \sqrt{\dfrac{\log(2N/\nu)}{\bar d_i}}
$$
and in particular taking $b = \sqrt{\dfrac{\log(2N/\nu)}{\bar d_i}}$
we get
\[
\pr\Big[\vert d_i-\bar d_i\vert> \sqrt{\log(2N/\nu)} \sqrt{\bar d_i} \Big]\leq \frac{\nu}{N} \, .
\]
Moreover, since $\bar d_i \le \bar d_{(N)}$, 
this further implies that 
\begin{equation} \label{eq:from-chernoff_new}
\pr\Big[\vert d_i-\bar d_i\vert> \sqrt{\log(2N/\nu)} \sqrt{\bar d_{(N)}} \Big]\leq \frac{\nu}{N} \, .
\end{equation}

For any $i$ such that $\bar \delta_i=0$, instead,
we recall that $\bar d_i = \sum_j W(X_{(i)}, X_{(j)})$, so that $\bar \delta_i=0$ implies
$W(X_{(i)}, X_{(j)}) = 0$ for all $j=1,\dots, N $. This also ensures that the edges at $i$ have probability zero, i.e., 
$$
\pr[d_i > 0]=0.
$$
This trivially implies that \eqref{eq:from-chernoff_new} holds true also when $\bar d_i = 0$.

Hence, for each $i$ (whether or not $\bar d_i = 0$), with probability at most $\nu / N$ we have
\[
\vert d_i-\bar d_i\vert > \sqrt{\log(2N/\nu)} \sqrt{\bar d_{(N)}} .
\]
Hence, with probability at least $1-\nu$ we have
\[ \vert d_i-\bar d_i\vert \le \sqrt{\log(2N/\nu)} \sqrt{\bar d_{(N)}} \quad
 \text{for all $i = 1,\dots, N$}. \]
Since $D_N-\bar D_N$ is diagonal, 
\begin{equation}\label{lem3_1_new}
\Vert D_N-\bar D_N\Vert_2=\max_{i=1,\ldots,N}\vert d_i-\bar d_i\vert
\leq \sqrt{\log(2N/\nu)} \sqrt{\bar d_{(N)}} 
\end{equation}
with probability at least $1-\nu$.

From Weyl's Theorem,
$\max_i |d_{(i)} - \bar d_{(i)}| \le \Vert D_N-\bar D_N\Vert_2$,
which ends the proof of \eqref{gamma_new}, 
recalling that $\bar d_{(N)} \le N$, $d_{(i)}=N\delta_{(i)}$ and $\bar d_{(i)}=N\bar \delta_{(i)}$.

For the second part of the lemma,
we have:
\begin{align*}
\Vert L_N-\bar L_N \Vert_2&=\Vert D_N-A_N-\bar D_N+\bar A_N \Vert_2 \\
 &\leq\Vert D_N-\bar D_N\Vert_2+\Vert \bar A_N-A_N \Vert_2.
\end{align*}
From \cite[Theorem 1]{chung2011spectra}, thanks to the assumption $\bar d_{(N)}> \frac{4}{9} \ln (2N/\nu)$, 
we have that with probability at least $1-\nu$:
\begin{equation}\label{prob_event}
\Vert A_N-\bar A_N \Vert_2\leq\sqrt{4\bar d_{(N)}\log(2N/\nu)}.
\end{equation}
By combining \eqref{lem3_1_new} and \eqref{prob_event} we have with probability at least $1-2 \nu$:
$$
\Vert L_N-\bar L_N \Vert_2\leq
\sqrt{\bar d_{(N)} \log(2N/\nu)} + 2 \sqrt{\bar d_{(N)} \log(2N/\nu)} =
3 \, \sqrt{\bar d_{(N)} \log(2N/\nu)}.
$$
By using Weyl's Theorem and considering the normalized eigenvalues we get:
$$
\max_i \vert \mu_i-\bar\mu_i \vert
\leq 3  \frac{1}{N} \sqrt{\bar d_{(N)} \log(2N/\nu)}.
$$
Finally, since $\bar d_{(N)}\leq N$, we get the desired result. 
\end{proof}

\begin{remark}[On the assumption $\bar d_{(N)}> \frac{4}{9} \ln (2N/\nu)$]  \label{remark-deltaN}
\phantom{a}
\begin{itemize}
\item     If $W$ has infimum $\eta_W>0$, then the assumption 
    $\bar d_{(N)}> \frac{4}{9} \ln (2N/\nu)$ is surely satisfied
    for all sufficiently large $N$, irrespective of the value of the latent variables $X_1, \dots, X_N$. Indeed, all degrees grow linearly with $N$, since for any $i$
    \[ {\bar d}_i = \sum_{j=1}^N W(X_i, X_j) \ge \eta_W N \,. \] 
For this reason, the assumption $\bar d_{(N)}> \frac{4}{9} \ln (2N/\nu)$ was not mentioned in the statement of Lemma~3, since it was implied by the assumptions $\eta_W>0$ and large enough $N$ considered therein.
\item 
    If $W$ is piecewise Lipschitz and satisfies the  corrected version of Eq.s~(4a)-(4b) in the definition of `large enough $N$' in Definition~2.2, as discussed in Sect.~\ref{sect:typos}, i.e., satisfies
    \begin{itemize}
        \item[(i)]   $2 b_N < \min_{k \in \{1, \dots, K+1\}} (\alpha_k - \alpha_{k-1})$,
        \item[(ii)]  $\frac{1}{N}\log\left(\frac{2N}{\nu}\right)+b_N (2K+3L)
        <\max_{x}d_W(x)$
    \end{itemize}
    with     $ b_N = \frac{1}{N} + \sqrt{\frac{8 \log (N/\nu)}{N+1}}$,
     then \cite[Lemma 5]{avella2018centrality} ensures that $\bar d_{(N)}> \frac{4}{9} \ln (2N/\nu)$ 
    with probability at least $1 - \nu$; this holds for any $\nu \in (N e^{-N/5}, e^{-1})$.\\ 
    Notice that (ii) implies that $W$ is not identically zero.
\end{itemize}
\end{remark}

\subsection{Improvements} \label{sect:improvements}

Notice that Lemma~\ref{lemma3-new}  improves upon Lemma~3 in two ways:  we have removed the assumption that $W$ is bounded away from zero (that is, $W$ has infimum $\eta_W >0$), which replaced by a milder assumption, and tighter upper bounds are obtained, which do not involve $\eta_W$. 

The latter is easy to see by comparing the new definitions of $\gamma(N)$ and $\varphi(N)$ in \eqref{gamma_new}-\eqref{phi_new} with the following ones, that appeared  as Eq.s~(13)-(14) in  Lemma~3:
\[ \gamma(N)=  \sqrt{\frac{1}{\eta_W}}  \sqrt{\frac{ \log(2N/\nu)}{N}} , \quad
\varphi(N) = \left( \frac{1}{\eta_W} +2 \right) 
    \sqrt{\frac{\log(2N/\nu)}{N}} , \]
recalling that $0<\eta_W\leq 1$.\\

\begin{remark}[Improvement in the constants]
The new definitions of $\gamma(N)$ and $\varphi(N)$ given in 
\eqref{gamma_new}-\eqref{phi_new} can replace the old definitions in all results of the paper in which $\gamma(N)$ and $\varphi(N)$ appear.
When using the expressions \eqref{gamma_new}-\eqref{phi_new} in Theorem~2, we can modify the assumption stated in there as (22)  (or rather the corrected version of (22) discussed in Sect.~\ref{sect:typos}), which is used to ensure that $\gamma(N)< \eta_W$ and $ \varphi(N)<\eta_W$ so that the terms in the denominator are positive. Hence (22) can be replaced by the following simpler assumption:
\[  \frac{\log(2N/\nu)}{N} < \frac{\eta_W^2}{9}  \,.  \]
\end{remark}

\begin{remark}[Dispensing from $W$ bounded away from zero]
 Thanks to Lemma~\ref{lemma3-new} and Remark~\ref{remark-deltaN}, the assumption `$W$ with infimum $\eta_W>0$' (and the equivalent assumption `$W$ with minimum $\eta_W>0$', as it is sometimes phrased because the regularity of $W$ ensures existence of the minimum) can be removed from the following results: 
\begin{itemize}
    \item Theorem~1 (and hence also Remark~1),  and
    \item Lemma 4.  
\end{itemize}
Indeed, in such results, this assumption was only used in order to apply Lemma~3, which can be replaced by Lemma~\ref{lemma3-new}; the assumptions $W$ piecewise Lipschiz and $N$ large enough, that are already present in these results, ensure that the assumption on $\bar d_N$ is satisfied, as discussed in Remark~\ref{remark-deltaN}.

On the other hand, the following results still require the assumption `$W$ with infimum $\eta_W>0$': 
Proposition~3, Remark 3, Theorem 2, and Remark 4.
Indeed, these results rely on the bound $\bar \mu \geq \eta_W>0$ obtained from Lemma~5 and the assumption $\eta_W>0$.

Remark 2 and Proposition 4 deserve some more detailed discussion, see Sect.~\ref{subsect:prop4}

\end{remark}

\subsection{Remark 2 and Proposition 4}  \label{subsect:prop4}

Remark~2 is about the asymptotic behavior for $N \to \infty$ of the explicit bounds given in Lemma~3. For this reason, it holds under the same assumptions as Lemma~3.
When replacing Lemma~3 with Lemma~\ref{lemma3-new}, Remark~2 clearly holds under the same assumptions as Lemma~\ref{lemma3-new}. However, it is useful to replace the assumption $\bar d_{(N)}> \frac{4}{9} \log (2N/\nu)$ by some assumption that implies it (at least for large $N$), and that is easier to verify; in particular, it is useful to consider an assumption that depends on the graphon $W$ only. From Remark~\ref{remark-deltaN}, we know that we can choose the assumption that $W$ has infimum $\eta_W>0$, or the assumption that $W$ is piecewise Lipschitz. Another useful alternative, that will be crucial to obtain a correct proof of Proposition~4, is to assume that $W$ is continuous; this can be done thanks to the following technical lemma.


\begin{lemma}  \label{lemma-deltaN}
    Assume that there exists $\eta>0$ and intervals $J_1, J_2 \subseteq [0,1]$ with length $0<\ell \leq 1$ such that $W(x,y) \geq \eta$ for all $(x,y) \in J_1 \times J_2$.
    Then, 
    \[ 
    \pr[\bar d_{(N)} \ge \eta N \ell/4]  \geq
    \left(1 - \exp{(-N \ell^2 / 4)} \right)
    \left( 1 - (1-\ell)^{\lfloor N/2 \rfloor}\right)
    \]
    and for $N \to \infty$
    \[\bar  d_{(N)} = \Theta(N) \text{ a.s.} \]
\end{lemma}
Notice that the existence of $\eta$, $J_1$ and $J_2$ as in the lemma is ensured for example if $W$ is continuous and not identically zero.
\begin{proof}
We will use the notation $[N]:= \{1 ,\dots, N \}$, 
$[N]_{\text{odd}}:= \{ n \in [N] : \text{$n$ odd}  \} $ 
and $[N]_{\text{evn}}:= \{ n \in [N] : \text{$n$ even}  \} $.

Consider the following events:
\begin{itemize}
    \item $E_1 := \{ \exists i \in [N] \text{ s.t.\ } X_i \in J_1 \}$, and
    \item $E_{2,c} := \{ \exists S \subseteq [N] \text{ with $|S| \geq c N$} ,\text{ s.t.\ } X_j \in J_2 \, \forall j \in S \}$, where $c>0$ will be chosen later.
\end{itemize}
Recalling that $X_1,\dots, X_N$ are obtained as a reordering of $U_1, \dots, U_N$, with $U_i$ i.i.d.\ uniform random variables, it is obvious that 
\begin{itemize}
    \item $E_1 := \{ \exists i \in [N] \text{ s.t.\ } U_i \in J_1 \}$, and
    \item $E_{2,c} := \{ \exists S \subseteq [N] \text{ with $|S| \geq c N$} \text{ s.t.\ } U_j \in J_2 \,\, \forall j \in S \}$.
\end{itemize}
Now consider the following events:
\begin{itemize}
    \item $\tilde E_1 := \{ \exists i \in [N]_{\text{evn}} \text{ s.t.\ } U_i \in J_1 \}$, and
    \item $\tilde E_{2,c} := \{ \exists S \subseteq [N]_{\text{odd}} \text{ with $|S| \geq c N$} \text{ s.t.\ } U_j \in J_2 \, \,\forall j \in S \}$.
\end{itemize}
Notice that $E_1 \supseteq \tilde E_1$ and $E_{2,c} \supseteq \tilde E_{2,c} $.
Moreover, $\tilde E_1$ and $\tilde E_{2,c} $ are independent events. Hence,
\[ \pr[E_1 \cap E_{2,c}] \geq 
\pr[\tilde E_1 \cap \tilde E_{2,c} ] = 
\pr[\tilde E_1] \, \pr[\tilde E_{2,c} ] \,.\]

We are interested in the events $E_1$ and $E_{2,c}$ because of the following simple remark:
if $X_i \in J_1$ and $X_j \in J_2$ for all $j \in S$, then
$\bar d_i \geq \eta |S|$, and this fact implies the inclusion
\[E_1 \cap E_{2,c} \subseteq \{ d_{(N)}  \geq  \eta c N\} \,, \]
which gives
\begin{equation}  \label{eq:proof-dN-ThetaN}
\pr[d_{(N)}  \geq  \eta c N]  \geq \pr[E_1 \cap E_{2,c}]
\geq  \pr[\tilde E_1] \,\pr[\tilde E_{2,c} ] \,.
\end{equation} 
We start by studying $\pr[\tilde E_1]$. Since $|[N]_{\text{evn}}| = \lfloor N/2 \rfloor$,
\[ \pr[\tilde E_1] = 1 - \pr[X_i \notin J_1 \, \forall i \in [N]_{\text{evn}}]
= 1 - (1 - \ell)^{\lfloor N/2 \rfloor} \,.\]
To study $\pr[\tilde E_{2,c} ]$, 
we define random variables
$Y_1, \dots , Y_N$, where $Y_j$ is equal to one if $U_j \in J_2$ and is equal to zero otherwise. 
Then, we define the random variable
$Z:= \sum_{j \in [N]_{\text{odd}}} Y_j$, i.e., $Z$ is the number of indexes $j \in [N]_{\text{odd}}$ such that $U_j \in J_2$.
Hence, $\tilde E_{2,c} = \{ Z \geq cN \}$.

Clearly $Y_1, \dots, Y_N$ are i.i.d.\ Bernoulli random variables with mean $\ell$.
Since $|[N]_{\text{odd}}| = \lceil N/2 \rceil$, $Z$ is a Binomial random variable $B(\lceil N/2 \rceil , \ell)$.
We can now use Hoeffding
inequality: 
for any $k < \E Z $, 
$\pr[Z \leq k] \leq  \exp(-2 (\E Z-k)^2 / \lceil N/2 \rceil)$.
Choosing $k = \frac{1}{2} \E Z  = \frac{\ell}{2} \left \lceil \frac{N}{2}\right \rceil $,
we get
\[  \pr[Z \leq \tfrac{\ell}{4}N]  
\leq \pr[Z \leq \tfrac{\ell}{2} \left \lceil \tfrac{N}{2}\right \rceil] 
\leq  \exp(- \tfrac{\ell^2}{2} \left \lceil \tfrac{N}{2}\right \rceil)  
\leq \exp(- \ell^2 N/ 4)  \,,\]
which ends the proof of the first claim, by considering
$c = \frac{\ell}{4}$ in 
$\{ Z \geq  c N \}   = \tilde E_{2,c}$ and in \eqref{eq:proof-dN-ThetaN}.

Then, for $N \to \infty$,  by Borel-Cantelli Lemma we obtain $\bar d_{(N)} = \Omega(N)$ a.s., and then it is also $\Theta(N)$  a.s.\ due to the trivial bound $\bar d_{(N)} \leq N$.
\end{proof}

\begin{remark}[Remark 2, improved] \label{remark2-new}
If $W$ has infimum $\eta_W > 0$, or if $W$ is piecewise Lipschitz, or if $W$ is continuous, then 
\[ |\mu_2 - \bar \mu_2| = O((\log(N)/N)^{1/2})  \text{ a.s.} \]
Indeed, under either of the first two assumptions, the assumption of Lemma~\ref{lemma3-new} is satisfied for large $N$ 
with probability at least $1-\nu$  
(by Remark~\ref{remark-deltaN}) and the above asymptotic behavior is obtained by applying Borel-Cantelli Lemma, same as in the original Remark 2.

We will now consider $W$ continuous.
Consider $\nu = N^{-\alpha}$ for some $\alpha > 1$. With this $\nu$, by Lemma~\ref{lemma-deltaN}, there exists $\ell>0$ and $\eta>0$ such that, for all sufficiently large $N$,
\begin{align*}
\pr[\bar d_{(N)} \ge \tfrac{4}{9} \log (2N/\nu)]
&\geq \pr[\bar d_{(N)} \ge \tfrac{4}{9} (1+\alpha) \log (2N)] \\
&\geq \pr[\bar d_{(N)} \ge \eta N \ell/4]  \\
&\geq   \left(1 - \exp{(-N \ell^2 / 4)} \right)
        \left( 1 - (1-\ell)^{\lfloor N/2 \rfloor}\right)\,.
\end{align*}
Moreover, by Lemma~\ref{lemma3-new} with the same $\nu = N^{-\alpha}$,
\[ \pr \Big[ \max_{i=1,\dots,N} |\mu_i - \bar \mu_i|  \leq \varphi(N) 
    \Big |   \bar d_{(N)} \ge \tfrac{4}{9} \log (2N/\nu)  \Big]  \geq 1-2\nu \,.\]
Hence,
\begin{align*}
\pr \Big[ \max_{i=1,\dots,N} |\mu_i - \bar \mu_i|  \leq \varphi(N) \Big]
& \geq  \pr \Big[ \{ \max_{i=1,\dots,N} |\mu_i - \bar \mu_i|  \leq \varphi(N) \}
    \cap  \{ \bar d_{(N)} \ge \tfrac{4}{9} \log (2N/\nu) \} \Big] \\
& \geq  (1-2\nu)
        \left(1 - \exp{(-N \ell^2 / 4)} \right)
        \left( 1 - (1-\ell)^{\lfloor N/2 \rfloor}\right) \,.
\end{align*}
Recalling that $\nu=N^{-\alpha}$ with $\alpha>1$, by Borel-Cantelli Lemma, for $N \to \infty$
\[  \max_{i=1,\dots,N} |\mu_i - \bar \mu_i|  = O(\varphi(N)) = O\left(\sqrt{\log N / N}\right)  \: \text{a.s.}\,.\]

To be precise, the above proof holds if $W$ is not identically zero. However, in case $W$ is identically zero, then the statement is trivially true, since in this case $\mu_2 = 0$ and $\bar \mu_2 = 0$.
\end{remark}

\bigskip


We can now turn our attention to Proposition~4. 
In \cite{vizuete2021laplacian}, Proposition~4 was stated without the assumption `$W$ with infimum $\eta_W>0$', but its proof relied on Lemma~3 (more precisely, on Remark~2, which was based on Lemma~3) and on results from \cite{von2008consistency}; Lemma~3 has the assumption $\eta_W>0$, and the paper \cite{von2008consistency}  states that $W$ is bounded away from zero among the general assumptions to be considered throughout the paper. However, we can now {\em confirm that Proposition~4 is correct}, i.e., it holds true without the assumption $\eta_W>0$. Concerning Remark~2, we can now use the new version, see Remark~\ref{remark2-new}. Concerning \cite{von2008consistency}, a careful look at all the proofs confirms that the assumption that $W$ is bounded away from zero is only used for the case of normalized Laplacian, while it is not needed for the un-normalized Laplacian, which is of interest here.\\

\section{Sampling with deterministic latent variables}
Finally, we also have another small extension, concerning deterministic sampling.
Section 3.4 in the paper already discusses how the results easily extend to the case of deterministic sampling, with the exception of Proposition~4. The reason is that Proposition~4 is built upon results in \cite{von2008consistency} that are obtained there for the stochastic sampling only. However, it is easy to extend such results to deterministic sampling. Indeed, in the results that we use from \cite{von2008consistency} and their proofs, the only claim that concerns the choice of stochastic $X_i$'s is \cite[Proposition~11]{von2008consistency} (Glivenko-Cantelli classes); in the case of deterministic latent variables $X_i = i/N$,
\cite[Proposition~11]{von2008consistency} can be replaced by the following result, which is analogous in the statement but much simpler to prove.

\begin{proposition}  $W$ a continuous graphon, $f:[0,1]\to \R$ continuous. 
Then
\[ \lim_{N \to \infty} \sup_{x \in [0,1]} 
    \left| \frac{1}{N} \sum_{j=1}^N W(x,X_j) -
           \int_0^1 W(x,y) \dy  \right| 
    = 0 \]
and
\[ \lim_{N \to \infty} \sup_{x \in [0,1]} 
    \left| \frac{1}{N} \sum_{j=1}^N W(x,X_j) f(X_j) -
           \int_0^1 W(x,y) f(y) \dy  \right| 
    = 0 \,. \]
\end{proposition}
\begin{proof}
$W$ being continuous on a compact, it is uniformly continuous:
for any $\eps > 0 $, there exists $\delta > 0$ such that
for all $(x_1,y_1)$ and $(x_2,y_2)$ with $\|(x_1,y_1) - (x_2,y_2)  \|<\delta$ 
we have $|W(x_1,y_1) - W(x_2,y_2)| < \eps$.

We consider the following simple upper bound:
\[ 
\left| \frac{1}{N} \sum_{j=1}^N W(x,X_j) - \int_0^1 W(x,y)  \dy \right| 
\leq 
\frac{1}{N} \sum_{j=1}^N  (u_j(x) - \ell_j(x)),
\]
where $u_j(x) := \max_{y \in I_j} W(x,y)$ and $\ell_j(x) := \min_{y \in I_j} W(x,y)$, 
with $I_j = [\frac{j-1}{N}, \frac{j}{N}]$.  
For all $N > 1/ \delta$, we have 
that $\|(x,y_1) - (x,y_2)\|<\delta$ for any $x$ and any $(y_1, y_2) \in I_j$ (any $I_j$), which ensures 
$(u_j(x) - \ell_j(x)) < \eps$ for any~$x$ and any~$j$.

We have obtained that for any $\eps>0$, there exists $\delta>0$ such that for all $N > \frac{1}{\delta}$, 
\[ \sup_{x\in [0,1]} \left| \frac{1}{N} \sum_{j=1}^N W(x,X_j) - \int_0^1 W(x,y)  \dy \right|  \leq  \eps \,,  \]
which concludes the proof of the first claim.

The proof of the second claim is the same, just considering the function $W(x,y)f(y)$ instead of $W(x,y)$.
\end{proof}

To ensure that Proposition~4 holds true in the case of deterministic latent variables, 
we use the above-mentioned modification of the results from \cite{von2008consistency}, together with the following simple adaptation of Lemma~\ref{lemma-deltaN}, from which it follows that Remark~\ref{remark2-new} remains true also when $X_i = i/N$.
\begin{lemma}
Consider $\bar G_N$ the weighted graph obtained with deterministic sampling, i.e., with adjacency matrix
$[\bar A_N]_{ij} = W(X_i,X_j)$, $X_i = i/N$.
Assume that there exists $\eta>0$ and closed intervals $J_1, J_2 \subseteq [0,1]$ with length $0<\ell \leq 1$ such that $W(x,y) \geq \eta$ for all $(x,y) \in J_1 \times J_2$. 
Then, for all  $N \geq 1/\ell$, 
\[ 
\bar d_{(N)} \ge \eta \lfloor \ell N \rfloor 
    \]
and for $N \to \infty$
\[\bar  d_{(N)} = \Theta(N) \,. \]
\end{lemma}
\begin{proof}
Simply notice that in a closed interval of length $\ell$ there are at least $\lfloor \ell N \rfloor$ points of the form $i/N$.
Hence, if $\lfloor \ell N \rfloor \geq 1$, there is at least one point $X_i \in J_1$ and there are  $\lfloor \ell N \rfloor$ points in $J_2$, so that $\bar d_i \geq \eta \lfloor \ell N \rfloor$.
\end{proof}

\bibliographystyle{plain}  

\end{document}